\documentclass{amsart}

\usepackage[latin1]{inputenc}
\usepackage{amsfonts}
\usepackage{amsmath}
\usepackage{amsthm}
\usepackage{amssymb}
\usepackage{latexsym}
\usepackage{enumerate}

\newtheorem{theorem}{Theorem}[section]
\newtheorem{lemma}[theorem]{Lemma}

\newtheorem{fact}[theorem]{Fact}

\newtheorem{definition}[theorem]{Definition}

\numberwithin{equation}{section}

\begin{document}

\newcommand{\cc}{\mathfrak{c}}
\newcommand{\N}{\mathbb{N}}
\newcommand{\Q}{\mathbb{Q}}
\newcommand{\R}{\mathbb{R}}

\newcommand{\PP}{\mathbb{P}}
\newcommand{\forces}{\Vdash}
\newcommand{\dom}{\text{dom}}
\newcommand{\osc}{\text{osc}}

\title{ON UNIVERSAL BANACH SPACES OF density CONTINUUM}

\author{CHRISTINA BRECH}
\thanks{The first author was supported by FAPESP fellowship (2007/08213-2), 
which is part of Thematic Project FAPESP (2006/02378-7). Part of the research was done
at the Technical University of \L\'od\'z where the first author was partially supported by Polish Ministry of
Science and Higher Education research grant N N201
386234.} 
\address{IMECC-UNICAMP, Caixa Postal 6065, 13083-970, Campinas, SP, Brazil}
\email{christina.brech@gmail.com}

\author{PIOTR KOSZMIDER}
\thanks{The second  author was partially supported by Polish Ministry of
Science and Higher Education research grant N N201
386234. Part of the research was done at the
University of Granada and supported by Junta de Andalucia and FEDER grant P06-FQM-01438.}

\address{Instytut Matematyki Politechniki \L\'odzkiej,
ul.\ W\'olcza\'nska 215, 90-924 \L\'od\'z, Poland}
\address{Departamento de An\'{a}lisis Matem\'{a}tico \\ Facultad de
 Ciencias \\ Universidad de Granada \\ 18071 Granada, Spain}
\email{\texttt{pkoszmider.politechnika@gmail.com}}

\date{}

\begin{abstract}
We consider the question whether there exists a Banach space $X$ of density continuum such that
every Banach space of density not bigger than continuum 
 isomorphically embeds into $X$ (called a universal
Banach space of density $\cc$).
It is well known that $\ell_\infty/c_0$ is such a space if we assume the continuum hypothesis.
However, some additional set-theoretic assumption is needed, as we prove in the main result of
this paper that
it is consistent with the usual axioms of set-theory
  that there is no universal Banach space of density $\cc$.
Thus, the problem of the existence of a universal Banach space of density $\cc$
is undecidable using the usual axioms of set-theory.

We also prove that it is consistent that there are universal
Banach spaces of density $\cc$, but $\ell_\infty/c_0$ is not among them.
This relies on the proof of the consistency of the nonexistence
of an isomorphic embedding of $C([0,\cc])$ into
 $\ell_\infty/c_0$.

\end{abstract}

\maketitle

\section{Introduction}

Let $\mathcal{C}$ be a class of Banach spaces.
We will use the following standard notion of a universal Banach space for $\mathcal{C}$: 
$X$ is universal (isometrically universal) for $\mathcal{C}$ if $X\in \mathcal{C}$ and
for any $Y\in\mathcal{C}$ there is an isomorphic (isometric) embedding $T:Y\rightarrow X$.

Similarly, one can define a universal Boolean algebra for a class of Boolean
algebras where the embedding is a monomorphism (injective homomorphism) of Boolean algebras.
For topological compact Hausdorff spaces it is natural in this context to
use the dual  notion: $K$ is universal for 
a class $\mathcal{T}$ of compact Hausdorff spaces if $K\in \mathcal{T}$ and
for any $L\in\mathcal{T}$ there is a continuous surjection $T:K\rightarrow L$.

Probably the most known and useful result about the existence of universal Banach spaces
is the classical Banach-Mazur theorem (Theorem 8.7.2 of \cite{semadeni}) which says that $C([0,1])$ is 
an isometrically universal space for the class of all separable Banach spaces.
On the other hand, Szlenk's theorem proved in \cite{szlenk} says that
there are no universal spaces for separable reflexive Banach spaces.

It is well-known that topological and Boolean algebraic objects translate into
Banach-theoretic ones but, in general, not vice-versa (see Chapters 7, 8 and 16 of \cite{semadeni}),
 in particular we have the following:

\begin{fact}\label{facttopologicalreduction} 
If there is a
universal Boolean algebra of cardinality $\kappa$ or a universal totally disconnected
compact space $K$
of weight $\kappa$ or a universal continuum of weight $\kappa$,
then there is an isometrically universal Banach space of density $\kappa$.
On the other hand, if there is an (isomorphically) universal Banach
space of density $\kappa$, then there is one of the form $C(K)$ for 
$K$ totally disconnected and there is one of the form $C(K)$ for $K$ connected.
\end{fact}
\begin{proof}

By the Stone duality, the existence of a universal
Boolean algebra of cardinality $\kappa$ is equivalent to the existence
of a universal totally disconnected compact space of weight $\kappa$.
Also, any compact Hausdorff topological space has a totally disconnected preimage
of the same weight (Proposition 8.3.5 of \cite{semadeni}) and so,
a universal totally disconnected compact space of weight $\kappa$
is also universal among all compact spaces of weight $\kappa$.

Recall that any Banach space $X$ is isometric to a subspace of $C(B_{X^*})$ 
(Proposition 6.1.9 of \cite{semadeni}), where 
$B_{X^*}$ is the dual unit ball of $X$ considered with the 
weak$^*$ topology, which is connected and has weight equal to
the density of $X$. 

To prove the first assertion, suppose there is a universal compact space $K$ (either totally
disconnected or connected) of weight $\kappa$. Given any Banach space of density $\kappa$ we get,
in both cases, a continuous surjection $\phi:K\rightarrow B_{X^*}$. 
The fact that $\phi$
induces an isometric embedding of $C(B_{X^*})$ into $C(K)$ (Theorem 4.2.2 of \cite{semadeni}) and
that $X$ can be isometrically embedded in $C(B_{X^*})$
implies that $C(K)$ is an isometrically universal Banach space of density $\kappa$.

Let us now prove the second assertion. If there is a 
 universal Banach space $X$ of density $\kappa$, as $X$ can be isometrically
embedded in $C(B_{X^*})$, any Banach space of density $\kappa$ can be embedded as well,
so that $C(B_{X^*})$ is a universal Banach space of density $\kappa$ (and
$B_{X^*}$ is connected). But $B_{X^*}$ has a continuous preimage $K$
which is totally disconnected and of the same weight, so that $C(K)$ is a 
universal Banach space of density $\kappa$ as well.
\end{proof}

In this paper we consider the question of the existence of 
a universal Banach space of density continuum, denoted $\cc$.

As Parovi\v cenko proved in \cite{parovicenko} that 
under CH, $\wp(\N)/Fin$ is a universal Boolean algebra of cardinality $\cc$,
the same hypothesis implies the existence of an isometrically universal Banach space
of density $\cc$, namely $C(K)$ where $K$ is the Stone space of $\wp(\N)/Fin$ (homeomorphic to $\beta\N\setminus\N$). 
Moreover, $C(K)$ is isometric to
$\ell_\infty/c_0$.
Conversely, using quite general model-theoretic methods,
Shelah and Usvyatsov showed in \cite{shelah} among others that it is consistent
that there is no isometrically universal Banach space of density $\cc$.
We can summarize these results as:

\begin{theorem}[\cite{parovicenko}, \cite{shelah}]
Assuming CH,  $\ell_\infty/c_0$ is an isometrically universal Banach space of density $\cc$.
On the other hand,
 it is consistent that there is no isometrically universal Banach space of density $\cc$.
\end{theorem}

The main result of our paper, proved in Section 2, shows that 
for the existence of a universal Banach space of density $\cc$
(where not only isometric isomorphisms are allowed as embeddings, but we allow all isomorphisms)
 some extra set-theoretic assumption is also necessary:

\begin{theorem}\label{theoremmainresult}
It is consistent that there is no universal Banach space of density $\cc$.
\end{theorem}

Our proof is quite inspired by the proof in \cite{dowhart} that 
it is consistent that there is no universal totally disconnected space (nor continuum) of density $\cc$.
However, the proof of \cite{dowhart} is allowed to
 rely on the fact that their embeddings (homeomorphisms) preserve set-theoretic
operations as well as the inclusion, which is not true in general for linear operators, i.e.,
$A\subseteq B$ does not imply that $T(\chi_A)\leq T(\chi_B)$, etc. 
Actually, by the Kaplansky theorem if there is
an order isomorphism between Banach spaces  $C(K)$ and $C(K')$
then $K$ and $K'$ are homeomorphic and so, there is 
exists an isometry of the Banach spaces (Theorem 7.8.1 of \cite{semadeni}).
Hence preserving the order
lies at heart of the difference between universal and isometrically universal
Banach spaces.

To overcome these difficulties we use a strong almost disjoint family
of subsets of $\omega_1$ (first constructed in \cite{baumgartner}),
i.e., a family $(X_\xi:\xi<\omega_2)$ of subsets of
$\omega_1$ such that $X_\xi\cap X_\eta$ is finite for distinct $\xi,\eta<\omega_2$.
Already the existence of such a family cannot be proved without extra set-theoretic assumptions.
Similar uses of almost disjoint families of $\N$ instead of $\omega_1$
were apparently initiated  by 
Whitley in his proof of the fact that $c_0$ is not complemented in $\ell_\infty$ (\cite{whitley}).

In Section 3 we prove that even if there are universal Banach spaces of density continuum,
$\ell_\infty/c_0$ does not have to be one of them.

\begin{theorem}
It is consistent that there are universal Banach spaces
of density $\cc$, but $\ell_\infty/c_0$ is not among them.
\end{theorem}

The model where this  takes place is the standard Cohen model. This time we
follow the main idea of the proof by Kunen in \cite{kunendoc} of the fact that in this
model the algebra $\wp(\N)/Fin$ does not contain well-ordered chains of length $\cc$. 
The main trick is to use the richness of the group of automorphisms of
Cohen's forcing which are induced by permutations of $\omega_2$. 
This allows us to prove that $\ell_\infty/c_0$ does not contain an isomorphic
copy of $C(K)$, where $K$ is the Stone space of a well-ordered chain of length $\cc$ (such a
$K$ is simply homeomorphic to $[0,\cc]=[0,\omega_2]$ with the order topology).
However, not all permutations which can be used in Kunen's proof would work 
in our argument.
\vskip 6pt

The terminology concerning forcing is based on \cite{kunen} and the one concerning
$C(K)$ spaces is based on \cite{semadeni}.
The results of this paper answer questions 5 and 6 from \cite{opit}.

\section{Nonexistence of a universal space of density $\cc$}

The model in which there will be no universal Banach spaces of density
$\mathfrak{c}$ is the model obtained by a product of
two forcings,
$\mathbb{P}_1$ and $\mathbb{P}_2$.
$\mathbb{P}_1$ is the c.c.c. forcing of Section 6 of  \cite{baumgartner} which adds
a strong almost disjoint family $(X_\xi: \xi<\omega_2)$ of uncountable subsets 
of $\omega_1$, that is, 
$X_\xi\cap X_\eta$ is finite whenever $\xi\neq\eta$ and $\mathbb{P}_2$ is the standard $\sigma$-closed and 
$\omega_2$-c.c. forcing for adding $\omega_3$ subsets of $\omega_1$ with countable conditions
($Fn(\omega_3\times\omega_1, 2,\omega_1)$ of Definition 6.1 of \cite{kunen}).
The ground model $V$ is a model of GCH.

\begin{definition}[\cite{baumgartner} Section 6]

Fix a family $(Y_\xi)_{\xi<\omega_2}$ of uncountable subsets of $\omega_1$ 
such that $Y_\xi \cap Y_\eta$ is countable for different $\xi, \eta \in 
\omega_2$ and let $\mathbb{P}_1$ be the partial order of functions $f$ whose
domain $\dom f$ is a finite subset of $\omega_2$, $f(\xi) \in [Y_\xi]^{<\omega}$
for every $\xi \in \dom f$ and, given $f, g \in \mathbb{P}_1$, put $f \leq g$ if

\begin{itemize}
\item $\dom g \subseteq \dom f$;
\item $g(\xi) \subseteq f(\xi)$ for every $\xi \in \dom g$;
\item $f(\xi) \cap f(\eta) = g(\xi) \cap g(\eta)$ for different $\xi, \eta \in \dom g$.
\end{itemize}

\end{definition}

Let us remind some properties of $\mathbb{P}_1$ which we will need.

\begin{lemma}\label{lemaP1}
The following assertions hold:

\begin{enumerate}[(a)]
\item $\mathbb{P}_1$  is c.c.c. of cardinality $\omega_2$.
\item $\mathbb{P}_1$   has precaliber $\omega_2$, 
that is, every set of cardinality $\omega_2$
has a centered subset of cardinality $\omega_2$.
\item $\mathbb{P}_1$ forces that there is a strong almost
disjoint family $(X_\xi: \xi<\omega_2)$ 
of uncountable subsets of $\omega_1$ (this is denoted 
$A(\aleph_1, \aleph_2,\aleph_1, \aleph_0)$ in \cite{baumgartner}).
\item If the ground model is a model of GCH, then
$\mathbb{P}_1$ forces $\cc=\omega_2$ and that GCH holds at other cardinals.
\end{enumerate}

\end{lemma}

\begin{proof}
For (a) and (c) see Section 6 of \cite{baumgartner}.
To prove (b), fix $(f_\alpha)_{\alpha \in \omega_2} 
\subseteq \mathbb{P}_1$. By the $\Delta$-system lemma, 
there is a subset $S \subseteq \omega_2$ of cardinality
 $\omega_2$ such that $(\dom f_\alpha)_{\alpha \in S}$
 is a $\Delta$-system of root $\Delta$. Since for each
 $\xi \in \Delta$ there are at most $\omega_1$ possibilities 
for $f_\alpha(\xi)$ (because $f_\alpha(\xi) 
 \in [Y_\xi]^{< \omega}$ and $|Y_\xi |=\omega_1$),
 by thinning out the family a finite number ($|\Delta|$)
 of times, we can assume without loss of generality that 
$f_\alpha|_\Delta =f_\beta|_\Delta$ for every $\alpha, 
\beta \in S$, which makes $(f_\alpha)_{\alpha \in S}$ a centered family. 
 For $\cc\leq\omega_2$
and GCH at other cardinals in (d), use the standard argument with nice-names (Lemma VII 5.13 of \cite{kunen}).
To obtain that
 $\cc\geq\omega_2$ in $V^{\PP_1}$, use 
Theorem 3.4 (a) of \cite{baumgartner}, where it is proved 
that under CH there is no strong almost disjoint family of size $\omega_2$.
\end{proof}

\begin{definition}
Let $\mathbb{P}_2$ be the forcing formed by partial
 functions $f$ whose domain $\dom f$ is a countable subset 
of $\omega_3 \times \omega_1$ and whose range is included 
in $2=\{0, 1\}$, ordered by extension of functions.
 Given a subset $A \subseteq \omega_3$, 
we denote by $\mathbb{P}_2(A)$ the forcing formed by the elements of $\mathbb{P}_2$ 
whose domain is included in $A\times\omega_1$.
\end{definition}

We summarize in the next lemma the properties of $\mathbb{P}_2$ which we will use.

\begin{lemma}\label{lemaP2} Assume GCH.
The following assertions hold:
\begin{enumerate}[(a)]
\item $\mathbb{P}_2$ is
isomorphic to $\mathbb{P}_2(A) \times \mathbb{P}_2(\omega_3\setminus A)$, for any $A \subseteq \omega_3$.
\item $\mathbb{P}_2$ is $\sigma$-closed and $\omega_2$-c.c.
\end{enumerate}
\end{lemma}
\begin{proof}
See Section VII 6 of \cite{kunen}.
\end{proof}

And finally we conclude some properties of the product $\mathbb{P}_1 \times \mathbb{P}_2$.

\begin{lemma}\label{lemaP1xP2} Assume GCH.
The following assertions hold:
\begin{enumerate}[(i)]
\item $\mathbb{P}_2$ forces that $\check{\mathbb{P}_1}$ is c.c.c.
\item $\PP_1\times\PP_2$ is $\omega_2$-c.c.
\item $\mathbb{P}_1 \times \mathbb{P}_2$ preserves cardinals and in
 $V^{\mathbb{P}_1\times \mathbb{P}_2}$ we have that $\cc=\omega_2$.

\item Let $\kappa=\omega_1$ or $\kappa=\omega_2$ and let  $A\subseteq \omega_3$.
If $X$ is in the model $V^{\PP_1\times\PP_2(A)}$ and $(Y_\xi:\xi<\kappa)\in V^{\PP_1\times\PP_2}$ is a
sequence of subsets of $X$ all of cardinality $\leq\kappa$, then
there is in $V$ a subset $A'$ of $\omega_3$ of cardinality $\leq\kappa$ such that
$(Y_\xi:\xi<\kappa)\in V^{\PP_1\times\PP_2(A\cup A')}$.

\item Suppose $A\subseteq \omega_3$ and $\beta\in \omega_3\setminus A$.
 If $X\in V^{\PP_1}$ is an uncountable
subset of $\omega_1$, then $X\cap G_\beta\not\in V^{\PP_1\times \PP_2(A)}$ 
where $G_\beta=G\cap{\PP}_2(\{\beta\})$ and $G$ is $\PP_1\times \PP_2$-generic over $V$.
\end{enumerate}
\end{lemma}

\begin{proof} In this proof we will be often 
 using the product lemma (Theorem VIII 1.4. of \cite{kunen}). It implies that
$\PP_1\times\PP_2$ can be viewed as the forcing iterations $\PP_1*\check{\PP}_2$
or $\PP_2*\check{\PP}_1$.

Note that in $V^{\PP_2}$ we have that $\check{\PP}_1=\PP_1$
and so, Lemma \ref{lemaP1} (a) implies (i).
For (ii) note that any product of an $\omega_2$-c.c. forcing and
a forcing which has precaliber $\omega_2$ is $\omega_2$-c.c., so
Lemma \ref{lemaP1} (b) and Lemma \ref{lemaP2} (b) imply (ii).

For (iii) note that $\omega_1$ is preserved by $\PP_2*\check{\PP}_1$ since 
it is preserved by $\PP_2$ by Lemma \ref{lemaP2} (b) and later by
$\check{\PP}_1$ by (i). Other cardinals are preserved by (ii).
In $V^{\PP_1}$ we have $\cc=\omega_2$ by Lemma \ref{lemaP1} (d).
It is also true in $V^{\PP_1\times\PP_2}$ since in $V^{\PP_1}$
the forcing $\check{\PP}_2$ is $\omega_1$-Baire and hence it does not add reals.

(iv) is a consequence of the standard factorization, as for example in
Lemma VIII 2.2. of \cite{kunen}, which can be applied by (ii) and Lemma \ref{lemaP2} (a).

For (v) consider in $V^{\PP_1\times \PP_2(A)}$
$$D_Y=\{p\in \check{\PP}_2(\{\beta\}): p^{-1}(\{1\})\cap X\cap dom(p)\not =Y\cap X\cap dom(p)\}$$
where $Y$ is any subset of $\omega_1$ in $V^{\PP_1\times \PP_2(A)}$.
Since for any $q\in \check{\PP}_2(\{\beta\})$ one can find a finite extension 
$p$ satisfying $p^{-1}(\{1\})\cap X\cap dom(p)\not =Y\cap X\cap dom(p)\}$ we may
conclude that $p\in \check{\PP}_2(\{\beta\})$ and hence, $D_Y$ is a dense subset
of $\check{\PP}_2(\{\beta\})$ which belongs to  $V^{\PP_1\times \PP_2(A)}$.
Now,  by the product lemma, $G_\beta$ as in (v) is
a $\check{\PP}_2(\{\beta\})$-generic over $V^{\PP_1\times\PP_2(A)}$ and so
 we may conclude (v).

\end{proof}

To prove the main result of this paper, we need a combinatorial lemma 
concerning measures over a Boolean subalgebra of $\wp(\omega_1)$.

\begin{lemma}\label{lemaCombinatorio}
Let $\mathcal{B}$ be a Boolean subalgebra of $\wp(\omega_1)$ which contains
all finite sets of $\omega_1$, 
$L$ its Stone space and let $(X_\gamma)_{\gamma \in \omega_2} 
\subseteq \mathcal{B}$ be a strong almost disjoint family.
 Given a family $(\mu_\xi)_{\xi \in \omega_1}$ in $C(L)^*$, there is $\gamma_0 \in \omega_2$ such that 
$$\forall \gamma \in (\gamma_0, \omega_2)\quad \forall
 \xi \in \omega_1 \quad \forall X\subseteq X_\gamma
\quad \text{if }X \in \mathcal{B}, \text{ then }\mu_\xi([X]) =  \sum_{\lambda \in X}
\mu_\xi([\{\lambda\}]),$$
where $[X]$ denotes the clopen subset of $L$ corresponding to $X$ by the Stone duality.
\end{lemma}
\begin{proof}
Let us first prove the following:
\vspace{0.2cm}

\noindent{\bf{Claim.}}
There is $\gamma' \in \omega_2$ such that
$$\forall \gamma \in (\gamma', \omega_2)\quad \forall \xi \in
 \omega_1 \quad \forall X\subseteq X_\gamma, \quad \text{if }X \in \mathcal{B}, 
\text{ then } \mu_\xi([X]) = \mu_\xi(\bigcup_{\lambda \in X} [\{\lambda\}]).$$

\noindent \textit{Proof of the claim.} Suppose by contradiction
 that the claim does not hold. Then, there is $A \subseteq \omega_2$ 
of cardinality $\omega_2$ such that for every $\gamma \in A$ there 
are $\xi_{\gamma} \in \omega_1$ and $Y_\gamma \subseteq X_{\gamma}$ such that
$$\mu_{\xi_\gamma}([Y_\gamma]) \neq \mu_{\xi_\gamma}(\bigcup_{\lambda \in Y_\gamma} [\{\lambda\}]).$$
Since there are at most $\omega_1$ possibilities for $\xi_\gamma$, 
we may assume without loss of generality that $\xi_\gamma =
 \xi$ for a fixed $\xi \in \omega_1$. Also, we can assume 
without loss of generality that there is a natural number $m$ such that for all $\gamma \in A$,
$$|\mu_\xi([Y_\gamma])- \mu_\xi(\bigcup_{\lambda \in Y_\gamma} [\{\lambda\}])|> \frac{1}{m}.$$

Let $n$ be a natural number greater than $m \cdot \Vert \mu_\xi\Vert$ and let
$\gamma_1,\dots,\gamma_n$ be different ordinals in $A$ such that 
$$\mu_\xi([Y_{\gamma_i}])- \mu_\xi(\bigcup_{\lambda \in Y_{\gamma_i}} [\{\lambda\}])$$
are either all positive or all negative. 

Since $(X_\gamma)_{\gamma \in \omega_2}$ is a strong almost disjoint
 family in $\mathcal{B}$, it follows that $(Y_\gamma)_{\gamma \in A}$ 
is also a strong almost disjoint family in $\mathcal{B}$ and so, putting 
$$E_\gamma = [Y_\gamma] \setminus \bigcup_{\lambda \in Y_\gamma} [\{\lambda\}],$$
we have that $(E_\gamma)_{\gamma \in A}$ is a pairwise disjoint 
family of Borel subsets of $L$. Then,
$$|\mu_\xi (\bigcup_{i=1}^n E_{\gamma_i})| =
 \sum_{i=1}^n |\mu_\xi([Y_{\gamma_i}])- 
\mu_\xi(\bigcup_{\lambda \in Y_{\gamma_i}} [\{\lambda\}])| \geq n \cdot \frac{1}{m} > \Vert \mu_\xi\Vert,$$
a contradiction, which concludes the proof of the claim.
\vskip 6pt
Let us now prove the lemma. Suppose by contradiction that the lemma does not hold.
 Then, there is $A \subseteq (\gamma', \omega_2)$ of cardinality
 $\omega_2$ such that for every $\gamma \in A$ there 
are $\xi_{\gamma} \in \omega_1$ and $Y_\gamma \subseteq X_{\gamma}$ such that
$$\mu_{\xi_\gamma}([Y_\gamma]) \neq \sum_{\lambda \in Y_\gamma} \mu_{\xi_\gamma}([\{\lambda\}]).$$
By the previous claim, we conclude that
$$\mu_{\xi_\gamma}(\bigcup_{\lambda \in Y_\gamma} [\{\lambda\}])
 \neq \sum_{\lambda \in Y_\gamma} \mu_{\xi_\gamma}([\{\lambda\}]).$$
Since there are at most $\omega_1$ possibilities for $\xi_\gamma$, 
we may assume without loss of generality that $\xi_\gamma = \xi$ for a fixed
 $\xi \in \omega_1$. Also, we can assume without loss of generality
 that there is a natural number $m$ such that for all $\gamma \in A$,
$$|\mu_\xi(\bigcup_{\lambda \in Y_\gamma} [\{\lambda\}])- 
\sum_{\lambda \in Y_\gamma} \mu_\xi([\{\lambda\}])|> \frac{1}{m}.$$

Fix $\gamma \in A$. Let $Z_\gamma = \{\lambda \in Y_\gamma: 
\mu_\xi([\{\lambda\}]) \neq 0\}$ and $W_\gamma = Y_\gamma 
\setminus Z_\gamma$. $Z_\gamma$ is a countable set and since $\mu_\xi$ is $\sigma$-additive,
$$\mu_\xi(\bigcup_{\lambda \in Z_\gamma} [\{\lambda\}]) = \sum_{\lambda \in Z_\gamma} \mu_\xi([\{\lambda\}]).$$
Then, putting 
$$\delta_\gamma = \mu_\xi(\bigcup_{\lambda \in Y_\gamma}
 [\{\lambda\}])-\sum_{\lambda \in Y_{\gamma}} \mu_\xi([\{\lambda\}]),$$
using the fact that $([\{\lambda\}])_{\lambda \in \omega_1}$
 is a pairwise disjoint family and that $\mu_\xi([\{\lambda\}])=0$ 
for any $\lambda \in W_\gamma$, we get that 
\begin{displaymath}
\begin{array}{rl}
\delta_\gamma = & \mu_\xi(\displaystyle{\bigcup_{\lambda \in Z_\gamma}}
 [\{\lambda\}]) + \mu_\xi(\displaystyle{\bigcup_{\lambda \in W_\gamma}}
 [\{\lambda\}]) - \displaystyle{\sum_{\lambda \in Z_\gamma}} \mu_\xi([\{\lambda\}]) -
 \displaystyle{\sum_{\lambda \in W_\gamma}} \mu_\xi([\{\lambda\}])\\
= & \mu_\xi(\displaystyle{\bigcup_{\lambda \in Z_\gamma}} [\{\lambda\}]) -
 \displaystyle{\sum_{\lambda \in Z_\gamma}} \mu_\xi([\{\lambda\}]) +
 \mu_\xi(\displaystyle{\bigcup_{\lambda \in W_\gamma}} [\{\lambda\}])  
=  \mu_\xi(\displaystyle{\bigcup_{\lambda \in W_\gamma}} [\{\lambda\}]).
\end{array}
\end{displaymath}

Let $n$ be a natural number greater than $m \cdot \Vert \mu_\xi\Vert$
 and let $\gamma_1,\dots,\gamma_n$ be different ordinals in $A$ such that
 $(\delta_{\gamma_j})_{1 \leq j \leq n}$ are either all positive or all negative.
 Since $W_\gamma \subseteq Y_\gamma \subseteq X_\gamma$ and $(X_\gamma)_{\gamma \in \omega_2}$
 is strong almost disjoint, then for each $1 \leq j \leq n$, let $F_j$ be a finite subset of $W_{\gamma_j}$
 such that $(W_{\gamma_j} \setminus F_j)_{1 \leq j \leq n}$ are pairwise disjoint.
 Note that by the additivity of $\mu_\xi$, 
$$\mu_\xi( \bigcup_{\lambda \in W_{\gamma_j} \setminus F_j} [\{\lambda\}]) 
= \mu_\xi(\bigcup_{\lambda \in W_{\gamma_j}} [\{\lambda\}]) - \sum_{\lambda \in F_j} 
 \mu_\xi([\{\lambda\}]) = \mu_\xi(\bigcup_{\lambda \in W_{\gamma_j}} [\{\lambda\}]) = \delta_{\gamma_j},$$
since $\mu_\xi([\{\lambda\}]) = 0$ for every $\lambda \in W_\gamma$. Then,
$$|\mu_\xi(\bigcup_{j=1}^n \bigcup_{\lambda \in W_{\gamma_j} \setminus F_j} [\{\lambda\}])| =
|\sum_{j=1}^n \mu_\xi(\bigcup_{\lambda \in W_{\gamma_j} \setminus F_j} [\{\lambda\}])| 
=|\sum_{j=1}^n  \delta_{\gamma_j}| \geq n \cdot
\frac{1}{m} > \Vert \mu_\xi\Vert,$$
which is a contradiction and completes the proof of the lemma.
\end{proof}

For the sake of the proof of the main result, let us adopt the following notation:
if $\mathcal{A}$ is a Boolean algebra, then $C_\Q({\mathcal A})$ is
the set of all formal linear combinations of elements of $\mathcal{A}$
with rational coefficients. If $\mathcal{A}\subseteq \mathcal{B}$
are Boolean algebras, then $C_\Q({\mathcal A})$ can be identified with
a (nonclosed) linear subspace of $C(K)$, where $K$ is the Stone space of $\mathcal{B}$.
If $\mathcal{A}=\mathcal{B}$, then the subspace is norm-dense by the Stone-Weierstrass theorem.
The closure $\overline{C_\Q(\mathcal{A})}$ will mean the norm closure.
We can talk about linear bounded functionals $\nu$ defined on the spaces 
$C_\Q({\mathcal A})$, which correspond to  finitely additive bounded
measures on $\mathcal{A}$ (see Section 18.7 of \cite{semadeni}). If $\mathcal{A}=\mathcal{B}$,
 they have unique extensions to 
continuous linear functionals on $C(K)$, where $K$ is the Stone space of $\mathcal{A}$,
which can be interpreted as Radon measures on $K$.
It turns out that families of finitely additive measures viewed as functionals on 
$C_\Q({\mathcal A})$ can code all the information about operators between
Banach spaces $C(K)$. This is useful in the forcing context, since
if $\mathcal{A}_0\subseteq\wp(\omega_2)$ is  in an intermediate model, then such a measure could be
interpreted as a subset of  $\wp(\omega_2)\times\wp(\omega)$ which belongs to the intermediate model,
so that the factorization of Lemma \ref{lemaP1xP2} (iv) can be applied, 
while the corresponding Radon measure is at least as big as the
Stone space of the Boolean algebra.
On the other hand,
representing operators $T$ into
a $C(K)$ space as functions sending $x\in K$ to $T^*(\delta_x)$
in the dual space is quite classical (see Theorem VI 7.1. of \cite{dunford}).
Here $T^*:C^*(K)\rightarrow C^*(K)$ is the adjoint
operator of $T$ given by $T^*(\mu)(f)=\mu(T(f))$ where by the Riesz representation theorem
the elements of $C^*(K)$ are identified with the Radon measures on $K$.
\vskip 6pt

\begin{proof} of Theorem \ref{theoremmainresult}.

Let $V$ be a model of GCH. By Lemma \ref{lemaP1xP2}, we have that
 $\mathbb{P}_1\times \mathbb{P}_2$ preserves cardinals
 and in the extension $V^{\mathbb{P}_1\times \mathbb{P}_2}$
 we have $\cc=\omega_2$.

By Fact \ref{facttopologicalreduction}, it is enough to prove that there
is no universal Banach space of density $\cc$
which is of the form $C(K)$ where  $K$ is totally disconnected, i.e., where $K$ is
 the Stone space of a Boolean algebra.
In the extension $V^{\mathbb{P}_1 \times \mathbb{P}_2}$, take any Boolean algebra $\mathcal{A}$
 of cardinality $\omega_2=\cc$. We will prove that 
 $C(K)$ is not a universal Banach space
of density $\cc$, where $K$ is the Stone space of $\mathcal{A}$. 

Since $\mathcal{A}$ has cardinality $\omega_2$, we may
assume that $\mathcal{A}$ is
a subalgebra of $\wp(\omega_2)$ and so,
Lemma \ref{lemaP1xP2} (iv) applies to the sequence of its elements
and hence, there is $\alpha<\omega_3$ such that $\mathcal{A}\in
  V^{\mathbb{P}_1\times \mathbb{P}_2(\alpha)}$.

Let $\mathcal{B}$
 be the Boolean subalgebra of $\wp(\omega_1)$ 
 of all subsets of $\omega_1$
 which are in $V^{\mathbb{P}_1\times \mathbb{P}_2(\alpha+\omega_2)}$
 and let $L$ be its Stone space in $V^{\mathbb{P}_1\times \mathbb{P}_2}$.

We will prove that in $V^{\mathbb{P}_1\times \mathbb{P}_2}$
the space  $C(L)$ cannot be isomorphically embedded in $C(K)$, 
which will give  that $C(K)$ is not a universal Banach 
space of density $\cc$, concluding the proof. Suppose it can be
 isomorphically embedded and let us derive a contradiction.

Work in $V^{\mathbb{P}_1\times \mathbb{P}_2}$.

Let $T:C(L)\rightarrow C(K)$ be an isomorphic embedding 
and $T^{-1}: T[C(L)]\rightarrow C(L)$
its inverse. Let $\mathcal{B}_0\subseteq\mathcal{B}\subseteq \wp(\omega_1)$ be the Boolean algebra of all
finite and cofinite subsets of $\omega_1$.

\vskip 6pt
\noindent{\bf Claim 1.} There is a  Boolean algebra $\mathcal{A}_0\subseteq \mathcal{A}$
and a bounded sequence
of finitely additive bounded measures $(\nu_\xi:\xi\in\omega_1)$ on $\mathcal{A}_0$
such that 
\begin{enumerate}
\item $|\mathcal{A}_0|\leq \omega_1$,
\item If $\xi\in\omega_1$ and $\overline{\rho}_\xi\in C(K)^*$ is such that
$\overline{\rho}_\xi(\chi_{[a]})=\nu_\xi(a)$ for each $a\in \mathcal{A}_0$, then
for each $\lambda\in\omega_1$ we have 
$\overline{\rho}_\xi(T(\chi_{[\{\lambda\}]}))=\delta_\xi(\{\lambda\})$.
\end{enumerate}

\noindent{\it Proof of the claim}:

For each $f\in C(K)$ there is a countable subalgebra $\mathcal{A}_f\subseteq \mathcal{A}$
such that $f\in \overline{C_{\Q}(\mathcal{A}_f)}$, because we can approximate $f$ by 
finite linear combinations with rational coefficients
of characteristic functions of clopen sets. So, take
any $\mathcal{A}_0$ such that $\mathcal{A}_{T(\chi_{[\{\lambda\}]})}\subseteq \mathcal{A}_0$
for every $\lambda\in\omega_1$.

Let $\phi_\xi=(T^{-1})^*(\delta_\xi)$ be a bounded linear functional on
$T[C(L)]$ which corresponds to $\delta_\xi$ on $C(L)$. In particular, we have that
$$\phi_\xi(T(\chi_{[\{\lambda\}]}))=\delta_\xi(\{\lambda\}).$$

Of course, $||\phi_\xi||\leq ||(T^{-1})^*||$, so the sequence of $\phi_\xi$'s is bounded.
By the Hahn-Banach theorem, for any $\xi\in\omega_1$, $\phi_\xi$ has a
norm-preserving extension $\psi_\xi$ which
 is defined on $C(K)$. Finally, for $\xi\in\omega_1$, let $\nu_\xi$
be the finitely additive bounded measure (see Section 18.7 of \cite{semadeni})
 on $\mathcal{A}_0$ defined by
$$\nu_\xi(a)=\psi_\xi(\chi_{[a]}).$$

Now suppose that $\overline{\rho}_\xi\in C(K)^*$ is such that
$\overline{\rho}_\xi(\chi_{[a]})=\nu_\xi(a)$ for each $a\in \mathcal{A}_0$. Then, 
$\overline{\rho}_\xi(\chi_{[a]})=\psi_\xi(\chi_{[a]})$ for each $a\in \mathcal{A}_0$, and so, by
the linearity, the Stone-Weierstrass theorem and the continuity, we may conclude that
$\overline{\rho}_\xi|\overline{C_{\Q}(\mathcal{A}_0)}=\psi_\xi|\overline{C_{\Q}(\mathcal{A}_0)}$.
But by the choice of $\mathcal{A}_0$, we have that $T(\chi_{[\{\lambda\}]})\in \overline{C_{\Q}(\mathcal{A}_0)}$
for each $\lambda\in\omega_1$ and so,
$$\overline{\rho}_\xi(T(\chi_{[\{\lambda\}]}) = \psi_\xi(T(\chi_{[\{\lambda\}]})) = \phi_\xi (T(\chi_{[\{\lambda\}]})) = \delta_{\xi}(\{\lambda\}),$$
which concludes the proof of the claim.

\vskip 6pt

By Lemma \ref{lemaP1xP2} (iv),
 we can find $B \subseteq \omega_3 \setminus \alpha$ of cardinality 
$\omega_1$ such that $B \in V$ and the sequence $(\nu_\xi:\xi<\omega_1)$ and Boolean algebra $\mathcal{A}_0$
are in  $V^{\mathbb{P}_1\times \mathbb{P}_2(\alpha\cup B)}$.
Now working in $V^{\mathbb{P}_1\times \mathbb{P}_2(\alpha\cup B)}$,
apply Tarski's theorem (Proposition 17.2.9 of \cite{semadeni}) to extend $\nu_\xi$'s to norm-preserving 
finitely additive bounded measures $\rho_\xi$ on $\mathcal{A}$. 
\vskip 6pt

\noindent{\bf Claim 2.}
For every $g\in C_{\Q}(\mathcal{A})$, the sequence
$(\int gd\rho_\xi: \xi<\omega_1)$ belongs to $V^{\mathbb{P}_1\times \mathbb{P}_2(\alpha\cup B)}$.

\noindent{\it Proof of the claim}: Both the
algebra $\mathcal{A}$ and the sequence $(\rho_\xi:\xi<\omega_1)$ belong to
$V^{\mathbb{P}_1\times \mathbb{P}_2(\alpha\cup B)}$.
Hence $C_{\Q}(\mathcal{A})$ is in this model and 
 the evaluation
of the integrals follows their linearity and depends only on the values of
the measures on $\mathcal{A}$. This completes
the proof of the claim.

\vskip 6pt

Now work again in  $V^{\mathbb{P}_1\times \mathbb{P}_2}$.  Consider the strong almost disjoint family
  $(X_\gamma: \gamma<\omega_2)$ added by $\mathbb{P}_1$ (by Lemma \ref{lemaP1})
 and the adjoint operator $T^*: C(K)^*\rightarrow C(L)^*$. 
 For $\xi\in\omega_1$, let $\overline{\rho}_\xi$ be the unique functional on $C(K)$
extending $\rho_\xi$ (see Section 18.7 of \cite{semadeni}), so for
each $a\in \mathcal{A}_0$ we have
$$\overline{\rho}_\xi(\chi_{[a]})=\rho_\xi(a)=\nu_\xi(a).$$
By Claim 1, we have that $\overline{\rho}_\xi(T(\chi_{[\{\lambda\}]}))=\delta_\xi(\{\lambda\})$
for each $\xi,\lambda\in\omega_1$.

Now,
let $\mu_\xi$ be the
Radon measure on $L$ corresponding to the functional $T^*(\overline{\rho}_\xi)$ on $C(L)$. In particular, for each $\xi,\lambda\in\omega_1$ we have 
\begin{equation}\label{eq2}
\mu_\xi([\{\lambda\}]) = \overline{\rho}_\xi(T(\chi_{[\{\lambda\}]}))= 
\delta_\xi(\{\lambda\}).
\end{equation}

By Lemma \ref{lemaCombinatorio}, there is $\gamma\in\omega_2$
 such that for each $\xi\in\omega_1$, we have 
$$\mu_\xi([X])=\sum_{\lambda \in X}\mu_\xi([\{\lambda\}])$$
for every $X\subseteq X_{\gamma}$, $X \in \mathcal{B}$. Then, 
if $\beta <\alpha + \omega_2$ is such that $\sup (\alpha+\omega_2)
 \cap B < \beta$ (which certainly exists because $B$ has cardinality
 $\omega_1$) and $G_{\beta}$ is the
projection
of the $\PP_1\times\PP_2$-generic $G$ over $V$ on the
$\beta$-th coordinate in $\PP_2$, we have that 
\begin{equation}\label{eq1}
\mu_\xi([X_\gamma\cap G_{\beta}])=
\sum_{\lambda\in X_\gamma \cap G_{\beta}}\mu_\xi([ \{\lambda \}]). 
\end{equation}

Taking $f=T(\chi_{[X_\gamma\cap G_{\beta}]})$ and combining equalities (\ref{eq2}) and (\ref{eq1}), it follows that 
$$\overline{\rho}_\xi(f) = \mu_\xi([X_\gamma\cap G_{\beta}])= 
\displaystyle{\sum_{\lambda \in X_\gamma \cap G_{\beta}}} \delta_\xi(\{\lambda\}) =
\left\{ \begin{array}{rl}
1 & \text{if } \xi\in X_\gamma\cap G_{\beta}\\
0 & \text{if } \xi\not \in X_\gamma\cap G_{\beta}.
\end{array}\right.
$$ 
Now, take $g\in C_{\Q}(\mathcal{A})$ 
 such that  $||g-f||<1/3||(T^*)^{-1}||$.

Using the fact that
$$\Vert \overline{\rho}_\xi \Vert = \Vert \phi_\xi \Vert  =
 ||(T^*)^{-1}(\delta_\xi)|| \leq ||(T^*)^{-1}||\cdot \Vert \delta_\xi \Vert \leq ||(T^*)^{-1}||,$$ 
we get that $|\overline{\rho}_\xi(g-f)|<1/3$, hence 
$$\int gd\rho_\xi=\int g d\overline{\rho}_\xi \left\{\begin{array}{rl}
> 2/3 & \text{if }\xi\in X_\gamma\cap G_{\beta}\\
<1/3 & \text{if }\xi\not \in X_\gamma\cap G_{\beta}.
\end{array} \right.$$
Since the formulas $\int g\rho_\xi >2/3$ and $\int g\rho_\xi
 <1/3$ are absolute, by Claim 2 we conclude that 
$X_\gamma\cap G_{\beta}$ belongs to 
$V^{\mathbb{P}_1\times \mathbb{P}_2(\alpha \cup B)}$,
 which contradicts Lemma \ref{lemaP1xP2} (v) and concludes the proof. 

\end{proof}

\section{$\ell_\infty/c_0$ may not be among existing universal Banach spaces}

Our main purpose in this section is to prove that the Banach space 
$\ell_\infty/c_0$ may fail to be a universal space of
 density $\cc$ and at the same time there may
exist universal Banach spaces of
 density $\cc$. Actually, we prove that this situation takes place
 in the model obtained by adding $\omega_2$ Cohen 
reals to a model of GCH. Moreover, the reason why 
$\ell_\infty/c_0$ is not universal can be seen quite explicitly, namely in that model it contains no
 isomorphic copy of $C([0, \cc])$.

\begin{definition}
Let $\mathbb{P}$ be the forcing formed by partial 
functions $f$ whose domains $\dom f$ are finite subsets
 of $\omega_2 \times \omega$ and whose ranges are 
included in $2=\{0, 1\}$, ordered by extension of functions.
 \end{definition}

\begin{theorem} Assume $GCH$.
$\mathbb{P}$ forces that there is a universal Banach space of density $\cc$.
\end{theorem}
\begin{proof}
This is just the conjunction of the results of \cite{dowhart} and \cite{fremlinnyikos}. 
As noted in \cite{dowhart} at the beginning of Section 6,
in \cite{fremlinnyikos} it is proved 
that the existence of a $\cc$-saturated ultrafilter on $\N$ is equivalent to
the conjunction of Martin's Axiom for countable partial orders and
$2^{<\cc}=\cc$. The former holds in any model obtained by adding 
at least $\cc$ Cohen reals and the latter holds in any model
obtained by a c.c.c. forcing of size $\omega_2$ which adds $\omega_2$ reals
over a model of CH$+ 2^{\omega_1}=\omega_2$.
Hence, if we assume GCH, $\mathbb{P}$ forces that there is a  $\cc$-saturated ultrafilter on $\N$.
Now we use the observation included at the end of
Section 5 of \cite{dowhart} that the existence of a $\cc$-saturated ultrafilter on $\N$
implies that there is a universal continuum of weight $\cc$ and apply
Fact \ref{facttopologicalreduction}.
\end{proof}

Now we proceed to the proof of the fact that $\mathbb{P}$ forces 
that $\ell_\infty/c_0$ contains no
 isomorphic copy of $C([0, \omega_2])$. This is motivated by
a result and the proof of Kunen in \cite{kunendoc} saying that  $\mathbb{P}$ forces that 
$\wp(\N)/Fin$ has no well-ordered chains of length $\cc$.

\begin{definition}
A nice-name for an element of  $\ell_\infty$ 
is a name of the form 
$$\dot{f} = \bigcup_{{n,m} \in \N\times \N} \{\langle [[\check{n},\check{m}],\check{q}_{n,m}(p)], p \rangle: p\in A_{n,m}\},$$   
where $[[\check{n},\check{m}],\check{q}]$ stands for the canonical
name for an ordered pair whose first element is ordered pair $\langle\check{n}, \check{m}\rangle$
 and whose second element is $q$;
$A_{{n,m}}$'s are maximal antichains in $\mathbb{P}$; and $q_{n,m}:A_{n,m}\rightarrow \Q$ are functions.

Given a nice-name $\dot{f}$ for an element of $\ell_\infty$ as above, we define the support of $\dot{f}$ by
$$supp (\dot{f}) = \bigcup \{\dom (p):  p \in A_{n,m}\}.$$
\end{definition}

Thus, formally the value of a nice-name for an element of
$\ell_\infty$ is a function  $f:\N\times\N\rightarrow\Q$. This can be
treated as a code for an element of $\ell_\infty$ for example 
if we associate with such an $f$ the element
of $\ell_\infty$ (formally a subset of $\R^\N$) equal to $\lim_{m\rightarrow\infty}f(n,m)$ at $n$.

\begin{theorem} Assume CH.
$\mathbb{P}$ forces that there is no isomorphism of $C([0, \omega_2])$ into $\ell_\infty/c_0$. 
\end{theorem}

\begin{proof} Assume CH
 and suppose that there was in $V^\mathbb{P}$
 an isomorphism of $C([0, \omega_2])$ into $\ell_\infty/c_0$:
 let $\dot{T}$ be a name for it. Fix $k \in \mathbb{N}$  and $q\in \Q$
such that  $\mathbb{P} \Vdash \Vert \dot{T} \Vert \leq {\check q}$ and
$\mathbb{P} \Vdash {\check q}\cdot \Vert \dot{T}^{-1} \Vert < k-1$. 

For each $\alpha < \omega_2$, let $\dot{f}_\alpha$ be 
a nice-name for an element of $\ell_\infty$ such that
 $\mathbb{P} \Vdash [\dot{f}_\alpha]_{c_0} = 
\dot{T}(\check{\chi}_{[0, \alpha]})$, where $[\dot{f}_\alpha]_{c_0}$ 
denotes the equivalence class of $\dot{f}_\alpha$ in 
$\ell_\infty/c_0$. Let $A_\alpha = supp (\dot{f}_\alpha)
 \subseteq \omega_2$. 

By CH and the $\Delta$-system lemma, there is $X \subseteq \omega_2$ of cardinality $\omega_2$
such that $(A_\alpha)_{\alpha \in X}$ form a 
$\Delta$-system and $\alpha\in A_\alpha$. Then, using standard arguments we
 may assume w.l.o.g. that whenever $\alpha<\beta$ and $\alpha,\beta\in X$, there
is an order-preserving function $\sigma_{\alpha,\beta}: A_\alpha\rightarrow A_\beta$
such that $\sigma_{\alpha,\beta}(\alpha)=\beta$, 
which is constant on $A_\alpha\cap A_\beta$ and such that
it lifts up to an isomorphism 
$\pi_{\alpha,\beta}:\PP(A_\alpha\cup A_\beta)\rightarrow \PP(A_\alpha\cup A_\beta)$ such that
$\pi_{\alpha,\beta}^2=Id_{\PP(A_\alpha\cup A_\beta)}$,
$\pi_{\alpha,\beta}^*(\dot{f}_\alpha)=\dot{f}_\beta$ and 
where $\pi^*$ is the lifting of $\pi$ to the $\PP$-names as in Definition VII 7.12. of \cite{kunen}.
As any finite permutation is a composition of cycles, for any finite $F\subseteq X$ and any
permutation $\sigma:F\rightarrow F$ there is a permutation
${\pi}_\sigma:\omega_2\rightarrow\omega_2$ such that 
\begin{equation}
\pi^*_\sigma({\dot f}_{\alpha})={\dot f}_{\sigma(\alpha)}
\end{equation}

Now, let $\sigma$ be a permutation of $\omega_2$
 with the following property: there are $\alpha_{1}
 < \dots < \alpha_{{2k}} <\omega_2$ all in $X$ such that 

\begin{equation}\label{1}
\sigma(\alpha_{{2i-1}}) = \alpha_i \text{ and } \sigma(\alpha_{{2i}})= 
\alpha_{{2k-(i-1)}}, \text{ for all } 1 \leq i \leq k.
\end{equation}
By Lemma VII 7.13 (c) of \cite{kunen} and (3.1), for any formula $\phi$ and
any permutation $\sigma$ of $X$ we have

\begin{equation}\label{2}
\\P
\Vdash \phi(\dot{f}_{\alpha_1},..., \dot{f}_{\alpha_{2k}})\ \ \  \hbox{\rm iff}\ \ \
\\P
\Vdash \phi(\dot{f}_{\sigma(\alpha_1)},..., \dot{f}_{\sigma(\alpha_{2k})}).
\end{equation}

Notice that $((\alpha_{2i-1}, \alpha_{2i}])_{1 \leq i \leq k}$ 
are pairwise disjoint clopen intervals and so, it follows that
$$ \Vert \sum_{i=1}^k \chi_{[0, \alpha_{2i}]} 
- \chi_{[0, \alpha_{2i-1}]}\Vert = \Vert \sum_{i=1}^k
 \chi_{(\alpha_{2i-1}, \alpha_{2i}]} \Vert =
\Vert \sum_{i=1}^k
 \chi_{(\alpha_{1}, \alpha_{2k}]} \Vert =1,$$
which implies that
$$\mathbb{P} \Vdash \Vert \sum_{i=1}^k \chi_{[0, \alpha_{2i}]} 
- \chi_{[0, \alpha_{2i-1}]}\Vert =1$$
and consequently,
$$\mathbb{P} \Vdash \Vert \sum_{i=1}^k [\dot{f}_{\alpha_{2i}}]_{c_0} 
- [\dot{f}_{\alpha_{2i-1}}]_{c_0} \Vert  = \Vert \dot{T}(\sum_{i=1}^k 
\chi_{[0, \alpha_{2i}]} - \chi_{[0, \alpha_{2i-1}]}) \Vert \leq \check{q}.$$
Then, by (\ref{2}),
\begin{equation}\label{3}
\mathbb{P} \Vdash \Vert \sum_{i=1}^k [\dot{f}_{\sigma(\alpha_{2i})}]_{c_0} -
 [\dot{f}_{\sigma(\alpha_{2i-1})}]_{c_0} \Vert \leq \check{q}.
\end{equation}

But  (\ref{1}) yields that $\alpha_k \in 
(\alpha_i, \alpha_{2k-(i-1)}]$ for every $1 \leq i < k$, so that 
\begin{equation*}
\begin{array}{cl}
\Vert \displaystyle{\sum_{i=1}^k} 
\chi_{[0, \sigma(\alpha_{2i})]} - 
\chi_{[0, \sigma(\alpha_{2i-1})]}\Vert & = 
\Vert \displaystyle{\sum_{i=1}^k} \chi_{[0, \alpha_{2k -(i-1)}]} - \chi_{[0, \alpha_i]}\Vert \\ 
   & = \Vert \displaystyle{\sum_{i=1}^k} \chi_{(\alpha_i, \alpha_{2k -(i-1)}]} \Vert \geq k-1,
\end{array}
\end{equation*}
and then, $\mathbb{P}$ also forces that
$$ \begin{array}{cl}
k-1 & \leq \Vert\displaystyle{\sum_{i=1}^k} \chi_{[0, \sigma(\alpha_{2i})]} - 
\chi_{[0, \sigma(\alpha_{2i-1})]}\Vert = \Vert \dot{T}^{-1} 
(\displaystyle{\sum_{i=1}^k} [\dot{f}_{\sigma(\alpha_{2i})}]_{c_0} -
 [\dot{f}_{\sigma(\alpha_{2i-1})}]_{c_0}) \Vert \\
& \leq \Vert \dot{T}^{-1} \Vert
 \cdot \Vert \displaystyle{\sum_{i=1}^k} [\dot{f}_{\sigma(\alpha_{2i})}]_{c_0}
 - [\dot{f}_{\sigma(\alpha_{2i-1})}]_{c_0} \Vert
\end{array}
$$
which implies that 
$$\mathbb{P} \Vdash \Vert \sum_{i=1}^k [\dot{f}_{\sigma(\alpha_{2i})}]_{c_0} -
 [\dot{f}_{\sigma(\alpha_{2i-1})}]_{c_0} \Vert \geq \frac{k-1}{\Vert \dot{T}^{-1} \Vert} > \check{q},$$
contradicting equation (\ref{3}) and concluding the proof.
\end{proof}

It is natural to ask if we can directly conclude the nonexistence
of an embedding of $C([0,\cc])$ into $\ell_\infty/c_0$ from
the fact that $\wp(\N)/Fin$ does not have well-ordered chains of length
$\cc$ in this model. This could be done, for example, if we could prove that
$Clop(K)$ has a well-ordered chain of length $\kappa$
whenever $C([0, \kappa])$ embeds isomorphically into $C(K)$, for
$K$ totally disconnected. 

However, this is not the case even if the embedding is
isometric. It is clear that $C([0, \kappa])$ isometrically
embeds into $C(K)$, where $K$ is the dual ball of
$C([0, \kappa])$ with the weak$^*$ topology.
Using Kaplansky's theorem (Theorem 4.49 of \cite{fabianetal}) saying
that any Banach space has countable tightness in the weak topology
and comparing the weak and the weak$^*$ topology in $K$, we can prove that 
in $K$ there are no sequences
$(U_\xi)_{\xi<\omega_1}$ of weak$^*$-open sets such that $\overline{U_\xi}\subseteq U_\eta$
for $\xi<\eta<\omega_1$.
Moreover, if $L$ is the standard totally disconnected preimage of $K$,
we can prove (using quite tedious and technical arguments which we do not include here) that $L$
has no uncountable well-ordered chains of clopen sets, but as $L$ is
a continuous preimage of $K$ we have an isometric copy
of $C([0,\kappa])$ inside $C(L)$.

Note, for example, that the situation with antichains instead of well-ordered chains is quite different.
For a totally disconnected $K$, 
$C(K)$ contains a copy of $C(L)$ where $L$ is the Stone
space of the Boolean algebra generated by an
uncountable pairwise disjoint family of elements (i.e., $C(L)$ is isomorphic to  $c_0(\omega_1)$)
if and only if 
the Boolean algebra
$Clop(K)$ contains a pairwise disjoint family of cardinality $\omega_1$ 
 (\cite{rosenthal}, Theorem 12.30 (ii) of \cite{fabianetal}). 

However, we still do not know if in the concrete case of $K=\beta\N\setminus\N$
it is possible not to have  well-ordered chains of length $\omega_2$ of clopen sets
and at the same time have an isomorphic (isometric) copy of $C([0,\omega_2])$
inside $C(\beta\N\setminus\N)\equiv l_\infty/c_0$.

\bibliographystyle{amsplain}

\end{document}